\documentclass[11pt,leqno]{amsart}
\usepackage{amsmath,amssymb,amsthm}
\usepackage{hyperref}

\newtheorem{theorem}{Theorem}[section]
\newtheorem{lemma}[theorem]{Lemma}

\newtheorem{proposition}[theorem]{Proposition}

\theoremstyle{definition}
\newtheorem{definition}[theorem]{Definition}
\newtheorem{example}[theorem]{Example}

\theoremstyle{remark}
\newtheorem{remark}[theorem]{Remark}
\numberwithin{equation}{section}
\newcommand{\abs}[1]{\lvert#1\rvert}

\def\fnote#1{\footnote}

\def\ignora#1{}
\def\n3#1{\left\vert  \! \left\vert \! \left\vert \, #1 \, \right\vert \!
  \right\vert \! \right\vert }

\newcommand{\Natural}{\mathbb N}

\newcommand{\Real}{\mathbb R}

\newcommand{\set}[1]{\left\{#1\right\}}

\newcommand{\cardinality}[1]{\abs{#1}}

\newcommand{\diam}{\mathop{\mathrm{diam}}\nolimits}

\newcommand{\norm}[1]{\left\Vert#1\right\Vert}

\newcommand{\closedball}[1]{B_{#1}}

\newcommand{\Free}{{\mathcal F}}
\newcommand{\Lip}{{\mathrm{Lip}}_0}
\newcommand{\conv}{\mathop\mathrm{conv}}

\begin{document}

\title{A characterisation of octahedrality in Lipschitz-free spaces}

\author{Anton\'in Proch\'azka}\thanks{The first author was partially supported by PEPS INSMI 2016.}
\address{Universit\'e Bourgogne Franche-Comt\'e, Laboratoire de Math\'ematiques UMR 6623, 16 route de Gray,
25030 Besan\c con Cedex, France}
\email{antonin.prochazka@univ-fcomte.fr}

\author{ Abraham Rueda Zoca }\thanks{The second author was supported by a research grant Contratos predoctorales FPU del Plan Propio del Vicerrectorado de Investigaci\'on y Transferencia de la Universidad de Granada and by Junta de Andaluc\'ia Grants FQM-0199.}
\address{Universidad de Granada, Facultad de Ciencias.
Departamento de An\'{a}lisis Matem\'{a}tico, 18071-Granada
(Spain) and Instituto de Matem\'aticas de la Universidad de Granada (IEMath-GR)} \email{ arz0001@correo.ugr.es}
\urladdr{\url{https://arzenglish.wordpress.com}}



\begin{abstract}
We characterise the octahedrality of Lipschitz-free space norm in terms of a new geometric property of the underlying metric space. 
We study the metric spaces with and without this property.
Quite surprisingly, metric spaces without this property cannot embed isometrically into $\ell_1$ and similar Banach spaces.
\end{abstract}

\maketitle
\markboth{ANTON\'IN PROCH\'AZKA AND ABRAHAM RUEDA ZOCA}{A CHARACTERISATION OF OCTAHEDRALITY IN LIPSCHITZ-FREE SPACES}

\section{Introduction}\label{sectintro}
\bigskip

The Lipschitz-free space $\Free(M)$ of a metric space $M$ is a Banach space such that every Lipschitz function on $M$ admits a canonical linear extension defined on $\Free(M)$ (for details see Section~\ref{s:preliminaries} and \cite{god2}).
This property leads naturally to the presence of isomorphic (sometimes even isometric) copies of $\ell_1$ in $\Free(M)$ when $M$ is infinite (see e.g. \cite{cdw,cj}).
At the same time, it is well known (\cite[Theorem II.4]{god}) that a Banach space $X$ contains an isomorphic copy of $\ell_1$ if and only if $X$ admits an equivalent octahedral norm.
Recall that the norm $\norm{\cdot}$ on a Banach space $X$ is \textit{octahedral} if, for every finite-dimensional subspace $Y\subseteq X$ and every $\varepsilon>0$, there exists $x\in S_X$ such that
$$\Vert y+\lambda x\Vert\geq (1-\varepsilon)(\Vert y\Vert+\vert \lambda\vert)$$
holds for every $y\in Y$ and every $\lambda\in\mathbb R$. Octahedrality can be regarded as a very strong negation of Fr\'echet differentiability of $\norm{\cdot}$ at every point of $X$.
It is thus natural to ask whether the Lipschitz-free space norm is always octahedral. 
This question has been treated recently by \cite[Theorem 2.4]{blrlip} who have shown that this is not necessarily always the case but it is sufficient, on the other hand, that the metric space be unbounded or non-uniformly discrete.

In this paper we introduce a new property of metric spaces, the \emph{long trapezoid property} (LTP), and show in Theorem \ref{t:circular} that the norm on $\Free(M)$ is octahedral if and only if $M$ has the LTP. This result thus joins the slowly growing group of isometric results on Lipschitz free spaces which permit to check a particular property of the metric space by looking only at the associated free space norm and vice versa. Namely we mean the characterization of the metric spaces whose free space is isometric to a subspace of $L_1(\mu)$ \cite{godard}, the characterization of the metric spaces whose free space is isometric to $\ell_1(\Gamma)$ \cite{dkp} and, to a certain extent, the characterization of the compact metric spaces whose free space enjoys the Daugavet property by \cite{ikw}.

It turns out that some easily identifiable classes of metric spaces always enjoy the LTP.
Apart from unbounded and non-uniformly discrete spaces that we have already mentioned, this is the case also for the infinite subsets of $\Real$-trees (Example \ref{ejepropiP}) and (more generally) infinite subsets of $\ell_1$. 
In an effort to understand the new LTP we study its permanence properties (Proposition~\ref{resulestaP}), taking an advantage of certain permanence properties for octahedral norms which might be new even in the general setting (Proposition~\ref{octa1sumanece}).
We give several examples of infinite metric spaces without the LTP (necessarily bounded and uniformly discrete). 
For instance, for every $\kappa \in \Real$ there is a $CAT(\kappa)$ space which admits a subset without the LTP, which contrasts with the above mentioned fact that subsets of $\Real$-trees have the LTP.
We also find examples of sets without the LTP in $\ell_p$ for every $1<p\leq \infty$ and in $c_0$ in Proposition \ref{p:ellp} but not in $\ell_1$. In fact, we prove that every infinite subset of a Banach space $X$ whose modulus of asymptotic uniform convexity is maximal enjoys the LTP (Proposition \ref{propauc}).

Finally, we give in Theorem \ref{teodifefrechet} a simple criterion in terms of the metric space $M$ for when some particular points of $\Free(M)$ might be points of Fr\'echet differentiability of the free space norm.

\section{Notation and preliminary results}\label{s:preliminaries}

\bigskip

We will consider only real Banach spaces. Given a Banach space $X$, we will denote the closed unit ball and the closed unit sphere by $B_X$ and $S_X$, respectively. We will also denote by $X^*$ the topological dual of $X$. By a \textit{slice of $B_X$} we will mean a set of the following form
$$S(B_X,f,\alpha):=\{x\in B_X: f(x)>1-\alpha\},$$
where $f\in S_{X^*}$ and $\alpha>0$. When $X$ is a dual Banach space, say $X=Y^*$, by a \textit{$w^*$-slice of $B_X$} we will mean the slice $S(B_X,y,\alpha)$ where $y\in S_Y$. A \textit{convex combination of slices of $B_X$} will be a set of the following form
$$\sum_{i=1}^n \lambda_i S_i,$$
where $\lambda_1,\ldots,\lambda_n$ are positive numbers such that $\sum_{i=1}^n \lambda_i=1$ and each $S_i$ is a slice of $B_X$. When $X$ is a dual Banach space, we will consider the concept of \textit{convex combination of $w^*$-slices of $B_X$} just replacing the concept of slice with the one of $w^*$-slice in the above definition. A subset $V\subseteq S_{X^*}$ is said to be \textit{norming} for $X^*$ if, for every $x^*\in X^*$, it follows $\Vert x^*\Vert=\sup\limits_{x\in V}x^*(x)$.

Given a pointed metric space $M$, that is, a metric space equipped with a distinguished point denoted by $0$, we will denote by $\Lip(M)$ the Banach space of all real Lipschitz functions defined on $M$ which vanish at $0$ under the standard Lipschitz norm
$$\Vert f\Vert:=\sup\left\{ \frac{\vert f(x)-f(y)\vert}{d(x,y)}: x\neq y\right\}.$$
It is known that $\Lip(M)$ is itself a dual Banach space (c.f. \cite{goka}). Indeed, for each $m\in M$, consider the function $\delta_m:\Lip(M)\to \mathbb R$ by $\delta_m(f)=f(m)$. Then $\delta_m\in \Lip(M)^*$ for all $m\in M$. Moreover, if we consider 
$$\mathcal F(M):=\overline{span}\{\delta_m: m\in M\}\subseteq \Lip(M)^*,$$
it follows that $\mathcal F(M)^*=\Lip(M)$. This space is known as the \textit{Lipschitz-free space over $M$} or simply the free space over $M$. 
It is usual to call the elements of $\Free(M)$ \emph{measures}. 
If $\mu=\sum_{i=1}^n a_i\delta_{x_i}$ with $x_i \in M\setminus\set{0}$ and $a_i\neq 0$ for all $1\leq i\leq n$, we will denote $supp(\mu):=\set{x_i:1\leq i\leq n}$, the \emph{support} of $\mu$.
Notice that the set $\left\{\frac{\delta_x-\delta_y}{d(x,y)}: x\neq y\right\}\subseteq S_{\mathcal F(M)}$ is norming for $\Lip(M)$. See \cite{god2} and references therein for background on Lipschitz-free spaces.

Given a Banach space $(X,\norm{\cdot})$, we will say that $\norm{\cdot}$ is octahedral if, for every $x_1,\ldots, x_n\in S_X$ and every $\varepsilon>0$, there exists $y\in S_X$ such that $\Vert x_i+y\Vert>2-\varepsilon$ holds for every $i\in\{1,\ldots, n\}$. It is proved in \cite[Proposition 2.1]{hlp} that this definition is equivalent to the one given by Godefroy introduced in Section \ref{sectintro}.

We will say that $(X,\norm{\cdot})$ has the \textit{strong diameter two property} (SD2P) if every convex combination of slices of $B_X$ has diameter two. When $X$ is a dual Banach space, we will consider the $w^*$ version of the above property defined in the natural way. It is known that a Banach space $(X,\norm{\cdot})$ is octahedral if, and only if, $(X^*,\norm{\cdot})$ has the $w^*$-SD2P \cite[Theorem 2.1]{blrocta}. Moreover, it is known that a Banach space admits an equivalent octahedral norm if and only if it contains an isomorphic copy of $\ell_1$ \cite{god}.

Given a Banach space $(X,\norm{\cdot})$ we say that $x\in X$ is a \textit{point of G\^ateaux differentiability of $X$} if the norm $\norm{\cdot}$ is G\^ateaux differentiable at $x$. By convexity of $\norm{\cdot}$, it is equivalent to the existence of the following limit 
$$\lim\limits_{t\rightarrow 0}\frac{\Vert x+th\Vert-\Vert x\Vert}{t}$$
for every $h\in X$.
We say that $x$ is a \textit{point of Fr\'echet differentiability of $X$} if the previous limit exists and it is uniform for $h\in S_X$. It is known that $x\in S_X$ is a point of G\^ateaux differentiability of $X$ if, and only if, there exists a unique $f\in S_{X^*}$ such that $f(x)=1$. 
Similarly, $x$ is a point of Fr\'echet differentiability of $X$ if, and only if, $\inf\limits_{\alpha>0}\diam(S(B_{X^*},x,\alpha))=0$, see \v Smulyan's lemma \cite[Corollary 7.20]{fhhmpz}.

\section{Main results}\label{s:main}

\bigskip

The next theorem is the main result of the paper.
The property of a metric space $M$ that appears in the point (\ref{hola3}) of the theorem will be called the \emph{long trapezoid property} (LTP) in the sequel. We can thus resume the theorem as ``a metric space $M$ has the LTP if, and only if, the norm on $\mathcal F(M)$ is octahedral.''
\begin{theorem}\label{t:circular}
For a metric space $M$ it is equivalent:
\begin{enumerate}
\item\label{hola1} The norm of $\Free(M)$ is octahedral.
\item\label{hola2} For each $\varepsilon>0$ and each finite subset $N \subset M$ there are points $u,v \in M$, $u \neq v$, such that every 1-Lipschitz function $f:N \to \Real$ admits an extension $\tilde{f}:M \to \Real$ which is $(1+\varepsilon)$-Lipschitz and satisfies $\tilde{f}(u)-\tilde{f}(v)\geq d(u,v)$. 
\item\label{hola3} For each finite subset $N\subseteq M$ and $\varepsilon>0$, there exist $u,v\in M, u\neq v$, such that
$$(1-\varepsilon)(d(x,y)+d(u,v))\leq d(x,u)+d(y,v)$$
holds for all $x,y \in N$.
\end{enumerate}
\end{theorem}

To prove this theorem we will need the following result, which brings to light the importance of norming subsets of Banach spaces with an octahedral norm.

\begin{proposition}\label{lemanormante}
Let $X$ be a Banach space with an octahedral norm and consider a norming subset $V\subseteq S_X$ for $X^*$. Then, given $x_1,\ldots, x_n\in S_X$ and $\varepsilon>0$ there exists $v\in V$ such that
$$\Vert x_i+ v\Vert>2-\varepsilon$$
holds for every $i\in\{1,\ldots, n\}$.
\end{proposition}

\begin{proof}
The proof will stronly rely on the ideas of \cite[Theorem 2.1]{blrocta}. Pick $x_1,\ldots, x_n\in S_X$ and $\varepsilon>0$. Consider the convex combination of weak-star slices of $B_{X^*}$ defined by $$C:=\frac{1}{n}\sum_{i=1}^n S(B_{X^*},x_i,\varepsilon).$$
As $diam(C)=2$ \cite[Theorem 2.1]{blrocta} then there are $\frac{1}{n}\sum_{i=1}^n f_i,\frac{1}{n}\sum_{i=1}^n g_i\in C$ such that
$$\left\Vert \frac{1}{n}\sum_{i=1}^n (f_i-g_i)\right\Vert >2-\frac{\varepsilon}{n}.$$
Since $V$ is norming for $X^*$ we can find $v\in V$ such that
$$\frac{1}{n}\sum_{i=1}^n (f_i-g_i)(v)>2-\frac{\varepsilon}{n}.$$
It follows that $f_i(v)-g_i(v)>2-\varepsilon$ 
and consequently $f_i(v)>1-\varepsilon$ holds for every $i\in\{1,\ldots, n\}$.
With this and since $f_i \in S(\closedball{X^*},x_i,\varepsilon)$ we have
$$\Vert x_i+v\Vert\geq f_i(x_i)+f_i(v)>1-\varepsilon+
1-\varepsilon=2-2\varepsilon$$
for every $i\in\{1,\ldots, n\}$, and the result follows.
\end{proof}

\begin{proof}[Proof of Theorem \ref{t:circular}]

(\ref{hola2}) $\Rightarrow$ (\ref{hola1}): Pick finitely-supported measures $\mu_1,\ldots, \mu_n\in S_{\mathcal F(M)}$ and $\varepsilon>0$. 
Define $N:=\set{0} \cup \bigcup\limits_{i=1}^n supp(\mu_i)$, which is a finite subset of $M$.
For each $i\in\{1,\ldots, n\}$ we can find $g_i\in S_{Lip_0(N)}$ such that $g_i(\mu_i)=\Vert \mu_i\Vert$. 
By (ii) we can find $u,v\in M, u\neq v$ such that, for each $i\in\{1,\ldots, n\}$, there exists $f_i\in Lip_0(M)$ such that $f_i=g_i$ on $N$, $f_i(u)-f_i(v)\geq d(u,v)$ and $\Vert f_i\Vert\leq 1+\varepsilon$. 
Pick $i\in\{1,\ldots, n\}$. 
Now
$$
\left\Vert\mu_i+\frac{\delta_u-\delta_v}{d(u,v)} \right\Vert\geq \frac{f_i(\mu_i)+\frac{f_i(u)-f_i(v)}{d(u,v)}}{1+\varepsilon}>\frac{g_i(\mu_i)+1}{1+\varepsilon}=\frac{\Vert \mu_i\Vert+1}{1+\varepsilon}.
$$
Consequently, the norm of $\mathcal F(M)$ is octahedral, as desired.\\
(\ref{hola1}) $\Rightarrow$ (\ref{hola3}): Pick a finite subset $N\subseteq M$ and $\varepsilon>0$. Since $\mathcal F(M)$ has an octahedral norm we can find, using of Proposition \ref{lemanormante}, two elements $u\neq v\in M$ such that
\[
\left\Vert \frac{\delta_x-\delta_y}{d(x,y)}+ \frac{\delta_u-\delta_v}{d(u,v)} \right\Vert>2-\varepsilon,
\]
holds for every $x\neq y\in N$. Hence, given $x\neq y\in N$, there exists $f\in S_{Lip(M)}$ such that
$$\frac{f(x)-f(y)}{d(x,y)}+\frac{f(u)-f(v)}{d(u,v)}>2-\varepsilon.$$
This implies the following two conditions
$$\begin{array}{ccc}
\frac{f(x)-f(y)}{d(x,y)}>1-\varepsilon, & \mbox{and} & \frac{f(u)-f(v)}{d(u,v)}>1-\varepsilon.
\end{array}$$
Now, we have the following chain of inequalities:
$$1\geq \frac{f(x)-f(v)}{d(x,v)}=\frac{f(x)-f(y)+f(u)-f(v)+f(y)-f(u)}{d(x,v)}>$$
$$>\frac{(1-\varepsilon)d(x,y)+
(1-\varepsilon)d(u,v)-d(u,y)}{d(x,v)}.$$
Consequently
$$(1-\varepsilon)(d(x,y)+d(u,v))<d(x,v)+d(u,y).$$
Since $x\neq y\in N$ were arbitrary we conclude (\ref{hola3}).\\
(\ref{hola3}) $\Rightarrow$ (\ref{hola2}): Let $N \subset M$ finite and $\varepsilon>0$ be given. 
By the assumptions, there are $u,v \in M$, $u\neq v$, such that 
\[
\frac{1}{1+\varepsilon}(d(x,y)+d(u,v))\leq d(x,u)+d(y,v)
\]
for all $x,y \in N$.
Given a 1-Lipschitz function $f$ on $N$ we define $\displaystyle\tilde{f}(u)=\inf_{x\in N} f(x)+(1+\varepsilon)d(x,u)$, $\displaystyle\tilde{f}(v)=\sup_{x\in N\cup\set{u}} \displaystyle\tilde{f}(x)-(1+\varepsilon)d(x,v)$. 
Clearly $\tilde{f}$ is $(1+\varepsilon)$-Lipschitz on $N \cup \set{u,v}$ so it admits an $(1+\varepsilon)$-Lipschitz extension to the whole of $M$.
Since $N$ is finite, there exist $z\in N$ and $z' \in N \cup \set{u}$ such that $\tilde{f}(u)=f(z)+(1+\varepsilon)d(z,u)$ and $\tilde{f}(v)=\tilde{f}(z')-(1+\varepsilon)d(z',v)$.
If $z'=u$, we have $\tilde{f}(u)-\tilde{f}(v)=(1+\varepsilon)d(u,v)$.
If $z'\neq u$, we have 
\[
\begin{split}
\tilde{f}(u)-\tilde{f}(v)&=f(z)-f(z')+(1+\varepsilon)(d(z,u)+d(z',v))\\
&\geq f(z)-f(z')+\frac{1+\varepsilon}{1+\varepsilon}(d(z,z')+d(u,v))\geq d(u,v)
\end{split}
\]
which finishes the proof.
\end{proof}

\begin{remark}\label{remathick}
Let $M$ be a metric space and $0\leq r<1$. Note that, adapting the proof of Theorem \ref{t:circular}, it can be proved that each of the following assertion implies the next one:
\begin{enumerate}
\item \label{thick1} For every $\mu_1,\ldots, \mu_n\in S_{\mathcal F(M)}$ and every $\varepsilon>0$ there exists $u\neq v\in M$ such that
$$\left\Vert \mu_i+\frac{\delta_u-\delta_v}{d(u,v)}\right\Vert\geq 2-r-\varepsilon$$
holds for every $i\in\{1,\ldots, n\}$.
\item \label{thick2} For each finite subset $N\subseteq M$ and $\varepsilon>0$, there exist $u,v\in M, u\neq v$, such that
$$(1-r-\varepsilon)(d(x,y)+d(u,v))\leq d(x,u)+d(y,v)$$
holds for all $x,y \in N$.
\item \label{thick3} For each $\varepsilon>0$ and each finite subset $N \subset M$ there are points $u,v \in M$, $u \neq v$, such that every 1-Lipschitz function $f:N \to \Real$ admits an extension $\tilde{f}:M \to \Real$ which is $\frac{1}{1-r-\varepsilon}$-Lipschitz and satisfies $\tilde{f}(u)-\tilde{f}(v)\geq d(u,v)$.
\item \label{thick4} For every $\mu_1,\ldots, \mu_n\in S_{\mathcal F(M)}$ and every $\varepsilon>0$ there exists $u\neq v\in M$ such that
$$\left\Vert \mu_i+\frac{\delta_u-\delta_v}{d(u,v)}\right\Vert\geq 2-2r-\varepsilon$$
holds for every $i\in\{1,\ldots, n\}$.
\end{enumerate}
We do not know whether (\ref{thick3}) actually implies (\ref{thick1}). Moreover, notice that Theorem \ref{t:circular} is the particular case of the above implications whenever $r=0$. Finally, notice that assertion (\ref{thick1}) is equivalent to the fact that the \textit{Whitley's thickness index of $\mathcal F(M)$} is greater than or equal to $2-r$ (we refer to \cite{cps} and references therein for formal definitions and background on such index).

\end{remark}

It is time to give some examples of metric spaces enjoying the LTP. 
Notice that Proposition~\ref{propauc} provides a sweeping generalization of the point (\ref{eje3}) below but we postpone it till Section~\ref{sectidifefrechet} as its proof requires some preparatory work.

\begin{example}\label{ejepropiP}
Any of the following properties implies that a metric space $M$ has the LTP.
\begin{enumerate}
\item\label{eje1} $M$ is unbounded.
\item\label{eje2} $\inf\limits_{x\neq y}d(x,y)=0$.
\item\label{eje3} $M$ is an infinite subset of an $\mathbb R$-tree (for the formal definition and basic properties of an $\Real$-tree see~\cite{BH,dkp}).
\end{enumerate}
\end{example}

\begin{proof}
Pick a finite subset $N\subseteq M$ and $\varepsilon>0$. 
In order to prove (\ref{eje1}), consider $v=0$. 
Then, if $d(0,u)$ is large enough,  we have for every $x,y \in N$ that
$$\frac{d(x,y)+d(u,0)}{d(x,u)+d(y,0)}\leq \frac{1+\frac{d(x,y)+d(x,0)}{d(x,u)}}{1+\frac{d(y,0)}{d(x,u)}} <\frac{1}{1-\varepsilon}.$$

In order to prove (\ref{eje2}), let $\theta=\inf_{x\neq y \in N}d(x,y)$ and find $u,v \in M$, $u\neq v$, such that $d(u,v)<\frac{\varepsilon \theta}{2}$. Then, for every $x,y \in N$, we have
\[\begin{split}
d(x,y)+d(u,v)&\leq d(x,u)+d(y,v)+2d(u,v)\\
& \leq d(x,u)+d(y,v)+\varepsilon(d(x,y)+d(u,v)).
\end{split}
\]
This proves (\ref{eje2}).

Finally, in order to prove (\ref{eje3}), consider $M$ to be a subset of an $\mathbb R$-tree $T$. 
We can assume by the preceding cases that $M$ is bounded and uniformly discrete. 
Note that $\conv(N)=\bigcup_{x\neq y\in N} [x,y]$ is compact since $N$ is finite.
Let $\pi:T \to \conv(M)$ be the metric projection onto $conv(N)$.
Let $a=\inf\{d(x,y):x\neq y \in M\}$. 
By compactness, there exists $A \subset \conv(N)$ such that $\diam(A)<\varepsilon a$ 
and that $A':=\{x \in M\setminus \conv(N): \pi(x) \in A\}$ is infinite. 
Pick $u,v\in A'$ such that $u\neq v$. 
Then
\[
\begin{split}
d(x,y)+d(u,v)&\leq d(u,\pi(u))+d(v,\pi(v))+d(\pi(u),\pi(v))+d(x,y)\\
&\leq d(u,\pi(u))+d(x,\pi(u))+d(v,\pi(v))+d(y,\pi(v))\\
&+2d(\pi(u),\pi(v))= d(u,x)+d(y,v)+2d(\pi(u),\pi(v))\\
&\leq d(x,u)+d(y,v)+\varepsilon(d(x,y)+d(u,v))
\end{split}
\]
for all $x,y \in N$, so we are done.
\end{proof}

\begin{remark} Note that, using Theorem \ref{t:circular} and \cite[Theorem 2.4]{blrlip}, it can be proved that unbounded and non-uniformly discrete metric spaces have the LTP. However, the previous example provides a quite shorter proof of this fact, and thus of \cite[Theorem 2.4]{blrlip} combining the previous example with Theorem \ref{t:circular}.
\end{remark}

Let us now exhibit an example of an infinite metric space $M$ failing the LTP.
A more detailed study of such examples will be conducted in Section~\ref{sectidifefrechet}.
\begin{example}\label{ejenega}
Let $M:=\{0\}\cup\{x_n:n\in\mathbb N\}\cup\{z\}$ where $d(a,b)=1$ for every $a,b\in \{0\}\cup\{x_n:n\in\mathbb N\}$, $d(x_n,z)=1$ holds for all $n\in\mathbb N$ and $d(0,z)=2$. Now consider $N:=\{0,z\}$ and  $0<\varepsilon<\frac{1}{3}$. Then for every $u,v\in M$ with $u\neq v$ it is clear that $d(0,z)+d(u,v)\geq 3$ but case by case check shows that $\min\{d(0,u)+d(z,v),d(0,v)+d(z,u)\}\leq 2$.
Thus  $M$ fails the LTP.
\end{example}

We will apply Theorem \ref{t:circular} to prove two stability results for the LTP. But first, we have to state a preliminary result concerning the octahedrality in $\ell_1$-sums of Banach spaces. Although it is probably well known to specialists, we have not found any accurate reference for one of the implications, thus we include a proof for the sake of completeness.
\begin{proposition}\label{octa1sumanece}
Let $X$ and $Y$ be Banach spaces. Then the norm on $X\oplus_1 Y$ is octahedral if and only if the norm on $X$ or the norm on $Y$ is octahedral.
\end{proposition}

\begin{proof}
The sufficiency is proved in \cite[Proposition 3.10]{hlp}.
Let us prove the necessity.
We will assume that the norms of $X$ and $Y$ both fail to be octahedral and we will prove that the norm of $Z:=X\oplus_1 Y$ is not octahedral. In order to do that we will prove that $Z^*=X^*\oplus_\infty Y^*$ fails the $w^*$-SD2P. By assumptions both $X^*$ and $Y^*$ fail the $w^*$-SD2P, hence there are two convex combinations of weak-star slices of the following form
$$\begin{array}{cc}
C_1:=\frac{1}{m}\sum_{i=1}^m S(B_{X^*},x_i,\alpha), & C_2:=\frac{1}{n}\sum_{i=1}^n S(B_{Y^*},y_i,\alpha)
\end{array}$$
such that $\diam(C_1)<2$ and $\diam(C_2)<2$. Assume, with no loss of generality, that $n\geq m$, and define
$$C:=\frac{1}{n}\left(\sum_{i=1}^m S(B_{X^*},x_i,\alpha)\times S(B_{Y^*},y_i,\alpha)+\sum_{i=m+1}^n B_{X^*}\times S(B_{Y^*},y_i,\alpha)  \right).$$
Notice that $C$ is a convex combination of non-empty relatively weakly-star open subsets of $B_{Z^*}$. Since each non-empty relatively weakly-star open subset of $B_{Z^*}$ contains a convex combination of weak-star slices of $B_{Z^*}$ (see the proof of \cite[Lemma II.1]{ggms}), it is enough to prove that $\diam(C)<2$. To this aim pick $\frac{1}{n}\sum_{i=1}^n (x_i,y_i), \frac{1}{n}\sum_{i=1}^n (x_i',y_i')\in C$. Now
$$\left\Vert \frac{1}{n}\sum_{i=1}^n((x_i,y_i)-(x_i',y_i'))\right\Vert=\max\left\{ \left\Vert \frac{1}{n}\sum_{i=1}^n(x_i-x_i')\right\Vert, \left\Vert \frac{1}{n}\sum_{i=1}^n(y_i-y_i')\right\Vert\right\}.$$
Let us prove that both members of the above maximum are strictly smaller than $2$. On the one hand, notice that $\frac{1}{n}\sum_{i=1}^n y_i, \frac{1}{n}\sum_{i=1}^n y_i'\in C_2$, hence $\left\Vert \frac{1}{n}\sum_{i=1}^n(y_i-y_i')\right\Vert\leq \diam(C_2)<2$. On the other hand
$$\left\Vert \frac{1}{n}\sum_{i=1}^n (x_i-x_i')\right\Vert\leq \frac{1}{n}\left(\left\Vert \sum_{i=1}^m(x_i-x_i')\right\Vert+\sum_{i=m+1}^n\Vert x_i-x_i'\Vert\right).$$
Again, since $\frac{1}{m}\sum_{i=1}^m x_i, \frac{1}{m}\sum_{i=1}^m x_i'\in C_1$ we get that $\left\Vert \sum_{i=1}^m(x_i-x_i')\right\Vert\leq m \diam(C_1)$. So
$$\left\Vert \frac{1}{n}\sum_{i=1}^n (x_i-x_i')\right\Vert\leq \frac{1}{n}\left(m \diam(C_1)+(n-m)2\right)<2$$
which finishes the proof.\end{proof}

Now we will exhibit the announced stability result for the LTP.

\begin{proposition}\label{resulestaP}
Let $M$ be a metric space. Then:
\begin{enumerate}
\item \label{estapl1suma} Assume that $M$ is the $\ell_1$ sum of its two subsets, say $T_1,T_2$, i.e. $M=T_1\cup T_2$, $T_1\cap T_2=\{0\}$ and 
$$d(x,y)=d(x,0)+d(0,y)$$
for every $x\in T_1$ and every $y\in T_2$. Then, $M$ has the LTP if, and only if, $T_1$ or $T_2$ has the LTP.
\item \label{pasosubespa} If $M$ has the LTP and $N_1$ is a subset of $M$ such that $M\setminus N_1$ is finite, then $N_1$ has the LTP.
\end{enumerate}
\end{proposition}

\begin{proof}
(\ref{estapl1suma}) Notice that the assumptions imply that $\mathcal F(M)=\mathcal F(T_1)\oplus_1 \mathcal F(T_2)$. 
Now the result follows applying Theorem~\ref{t:circular} twice and Proposition~\ref{octa1sumanece} once in between.

(\ref{pasosubespa}) We assume without loss of generality that $0\in N_1$. Let us denote $N_2:=\set{0} \cup M\setminus N_1$.
Notice that if $M$ is either unbounded or non-uniformly discrete then so is $N_1$. So we will assume that $M$ is a bounded and uniformly discrete metric space. 
In this case the following retractions will be Lipschitz: 
\[
\begin{array}{ccc}
r_1(x)=
\begin{cases}
x&\mbox{ if }x\in N_1\\
0&\mbox{ if }x \in N_2
\end{cases}&\mbox{ and }
r_2(x)=
\begin{cases}
0&\mbox{ if }x\in N_1\\
x&\mbox{ if }x \in N_2
\end{cases}
\end{array}
\]
Clearly $r_1\circ r_2(x)=r_2\circ r_1(x)=0$ and so the unique linear extensions $\bar{r_i}:\Free(M) \to \Free(N_i)$ of $r_i$, $i=1,2$, are continuous linear projections such that $\ker r_1=\Free(N_2)$ and vice versa. It follows that $\Free(M)=\Free(N_1)\oplus \Free(N_2)$.
The norm on $\mathcal F(M)$ is octahedral by the hypothesis and Theorem \ref{t:circular}.
Since $\dim\Free(N_2)<\infty$, \cite[Theorem 3.9]{abrahamsen} implies that $\Free(N_1)$ is octahedral.
Now another application of Theorem \ref{t:circular} shows that $N_1$ has the LTP.
\end{proof}

\begin{remark}
The assumption in Proposition \ref{resulestaP} (\ref{pasosubespa}) of  $M\setminus N_1$ being finite can not be removed.
This can be seen easily by taking the $\ell_1$ sum of two infinite metric spaces, one enjoying and the other one failing the LTP, and applying Proposition~\ref{resulestaP}~(\ref{estapl1suma}).
\end{remark}

\section{Metric spaces failing the property LTP}\label{sectidifefrechet}

\bigskip

When a metric space $M$ fails the LTP, one might wonder whether this can be checked on a subset $N$ consisting of mere 2 points. The next example provides a negative answer.

\begin{example}
Consider $M:=\{\alpha,\beta,0,z\}\cup\{x_n:n\in\mathbb N\}$ whose distance is defined in the following way:
$$d(0,x_n)=d(x_n,z)=1, d(0,z)=2, d(\alpha,0)=d(\beta,0)=1, d(\alpha,\beta)=2,$$
$$ d(\alpha, x_n)=d(\beta,x_n)=2, d(\alpha,z)=d(\beta,z)=3\mbox{ and } d(x_n,x_m)=1.$$
Denote by $T:=\{0,z\}\cup\{x_n:n\in\mathbb N\}$. Then
$$\mathcal F(M)=\mathcal F(T)\oplus_1\mathcal F(\{0,\alpha\})\oplus_1\mathcal F(\{0,\beta\}),$$
so $\mathcal{F}(M)$ fails to have an octahedral norm and, consequently, $M$ fails the LTP. We will prove, however, that the condition of LTP holds for every subset of $M$ of cardinality $2$. To this aim, pick $a,b\in M$. Then we have three possibilities for $a$ and $b$:
\begin{enumerate}
\item $d(a,b)=3$. Then, up to re-labeling $a$ and $b$, $b=z$ and $a$ is either $\alpha$ or $\beta$. We will assume, with no loss of generality, that $a=\alpha$. Then, the choice $u=\beta$ and $v=0$ does the work.
\item $d(a,b)=2$. In this case, we still have two more possibilities:
\begin{enumerate}
\item $a=0$ and $b=z$. In this case it is enough to choose $u=\alpha, v=\beta$.
\item $b=x_n$ for certain $n\in\mathbb N$ and $a$ is either $\alpha$ or $\beta$. We assume, with no loss of generality, that $a=\alpha$. In this case $u=\beta$ and $v=0$ yields the desired condition.
\end{enumerate}
\item $d(a,b)=1$. In this case, choose $u\neq v$ points such that $d(u,v)=1$ and $u,v$ being different from $a$ and $b$, and the inequality trivially holds.
\end{enumerate}
\end{example}

In spite of the previous example, the failure of the LTP can be checked on subsets of two points when we restrict our attention to a suitable metric subspace. More precisely, we get the following result.

 \begin{proposition}\label{propiramsey}
Let $M$ be a metric space failing the LTP. Then there exists an infinite subspace $A \subset M$ such that, for some $\varepsilon>0$ and some $x,y \in A$, we have
\[
(1-\varepsilon)(d(x,y)+d(u,v))>\min\set{d(x,u)+d(y,v),d(x,v)+d(y,u)}
\]
for all $u,v \in A$.
\end{proposition}
\begin{proof}
There is $\varepsilon>0$ and a finite $N\subset M$ such that, for every couple $u\neq v \in M\setminus N$, we have
\[
(1-\varepsilon)(d(x,y)+d(u,v))>\min\set{d(x,u)+d(y,v),d(x,v)+d(y,u)}
\]
for some couple $x\neq y \in N$. 
Since there are only finitely many couples $x \neq y \in N$ a direct application of Ramsey's theorem gives that there exist $x_0\neq y_0 \in N$ and an infinite $A' \subset M$ such that 
\begin{equation}\label{eq:failure}
(1-\varepsilon)(d(x_0,y_0)+d(u,v))>\min\set{d(x_0,u)+d(y_0,v),d(x_0,v)+d(y_0,u)}
\end{equation}
for every $u\neq v \in A'$.
If $\set{x_0,y_0} \subset A'$, the result is true for $A:=A'$.
If not, we denote 
\[
A(x):=\set{z \in A': \mbox{ \eqref{eq:failure} fails for } u=x,v=z}.
\]
Note that 
\[
A(x_0):=\set{z \in A': d(y_0,z)\geq (1-\varepsilon)(d(x_0,y_0)+d(x_0,z))},
\]
that $A(x_0)\cap \set{x_0,y_0}=\emptyset$, and that similar properties hold for $A(y_0)$.
We put 
\[
A:=\set{x_0,y_0}\cup A'\setminus (A(x_0) \cup A(y_0)) .
\]
We claim that $A(x_0)$, resp. $A(y_0)$, is a singleton at most. 
In order to get a contradiction assume that $z\neq w \in A(x_0)$.
We have, without loss of generality, the following inequality
\[
\begin{split}
(1-\varepsilon)(d(x_0,y_0)+d(z,w))&>d(x_0,z)+d(y_0,w)\\
&\geq
d(x_0,z)+(1-\varepsilon)(d(x_0,y_0)+d(x_0,w))\\
&=(1-\varepsilon)(d(x_0,y_0)+d(x_0,z)+d(x_0,w))+\varepsilon d(x_0,z)\\
&\geq (1-\varepsilon)(d(x_0,y_0)+d(z,w))+\varepsilon d(x_0,z),
\end{split}
\]
which is absurd.
Hence $\cardinality{A(x_0)}\leq 1$. 
An identical proof shows that $\cardinality{A(y_0)}\leq 1$.  
It follows that $A$ is infinite which we wanted to prove.
\end{proof}

A prominent class of non-octahedral norms are the norms that admit a point of Fr\'echet differentiability. 
The main results of the paper imply that for the norm on $\Free(M)$ this can happen only when $M$ is uniformly discrete and bounded.
It is not difficult to come up with an example. 
Indeed, one can easily see that the norm on $\Free(M)$ is Fr\'echet differentiable at $\frac{\delta_z}{d(z,0)}$ when $M$ and $z$ are as in Example~\ref{ejenega}.
We wish to take a closer look at this phenomenon.

\begin{theorem}\label{teodifefrechet}
Let $M$ be a uniformly discrete bounded metric space. Consider $x_1,\ldots, x_n, y\in M$ and $\lambda_1,\ldots, \lambda_n\in\mathbb R^+$ such that $\sum_{i=1}^n \lambda_i=1$.  Define $\varphi:=\sum_{i=1}^n\lambda_i \frac{\delta_{x_i}-\delta_y}{d(x_i,y)}$. The following are equivalent:
\begin{enumerate}
\item\label{frechet2} $\varphi$ is a Fr\'echet differentiability point of $\mathcal F(M)$.
\item\label{linea2} Given $z\in M$ there exists $i\in\{1,\ldots,n\}$ such that 
$$d(x_i,y)=d(x_i,z)+d(z,y).$$
\item\label{gateaux2} $\varphi$ is a G\^ateaux differentiability point of $\mathcal F(M)$.
\end{enumerate}
\end{theorem}

\begin{proof}We will assume with no loss of generality that $y=0$.

(\ref{linea2})$\Rightarrow$(\ref{frechet2}). Pick $\varepsilon>0$ and $f\in B_{\Lip(M)}$ such that $\varphi(f)=\sum_{i=1}^n\lambda_i \frac{f(x_i)}{d(x_i,0)} >1-\frac{\varepsilon}{\min\limits_{1\leq i \leq n}\lambda_i}$. An easy convexity argument yields that $f(x_i)>(1-\varepsilon)d(x_i,0)$ for each $i\in \{1,\ldots, n\}$. Pick an element $z\in M$. By assumptions there exists $i\in\{1,\ldots, n\}$ such that $d(x_i,0)=d(x_i,z)+d(z,0)$. 
Now
\[
\begin{split}
d(z,0)\geq f(z)&\geq f(x_i)-\vert f(z)-f(x_i)\vert>(1-\varepsilon)d(x_i,0)-d(x_i,z)\\
&=(1-\varepsilon)(d(x_i,z)+d(z,0))-d(x_i,z)=d(0,z)-\varepsilon d(x_i,0)
\end{split}
\]
We thus have $\vert f(z)-d(z,0)\vert<\varepsilon d(x_i,0)<\varepsilon \diam(M).$
Consequently, one has
$$\Vert f-d(\cdot, 0)\Vert\leq C\norm{f-d(\cdot,0)}_\infty\leq \varepsilon C\diam(M)$$
where $C\geq 1$ is the constant of equivalence between the Lipschitz and the uniform norm on $\Lip(M)$ (we recall that $\Lip(M)$ is isomorphic to $\ell_\infty(M\setminus\set{0})$ as $M$ is uniformly discrete and bounded).
According to \v Smulyan lemma, $\varphi$ is a point of Fr\'echet differentiability (with $d(\cdot,0)\in \Free(M)^*$ being the differential).

(\ref{frechet2})$\Rightarrow$(\ref{gateaux2}) is obvious.

(\ref{gateaux2})$\Rightarrow$(\ref{linea2}). Assume that for some $z \in M$, (\ref{linea2}) does not hold for any $x_j$ and let us prove that (\ref{gateaux2}) does not hold either. To see that define $f_i:\{0,x_1,\ldots, x_n, z\}\to \mathbb R$ for $i=1,2$ as follows: $f_i(0)=0$, $f_i(x_j)=d(0,x_j)$ for every $j\in\{1,\ldots, n\}$, $f_1(z)=d(0,z)$ and $f_2(z)=\max\{-d(0,z),\max\limits_{1\leq i\leq n}d(x_i,0)-d(z,x_i)\}$. 
We have clearly $\|f_i\|=1$, $\langle f_i,\varphi\rangle=1$ and $f_1\neq f_2$. Indeed, by assumptions $f_2(z)<d(z,0)=f_1(z)$. Now the respective norm-one extensions $\tilde{f}_i$ of $f_i$, $i=1,2$ show that $\varphi$ is not a point of G\^ateaux differentiability. \end{proof}

Let $X$ be a Banach space and let $\norm{\cdot}$ be an equivalent non-octahedral norm on $X$. 
It is easily seen that there exists $\varepsilon>0$ such that every norm $\abs{\cdot}$ which satisfies
\[
\frac{1}{1+\varepsilon}\norm{x}\leq \abs{x}\leq(1+\varepsilon)\norm{x}
\]
is non-octahedral.
Let now $(M,d)$ be a bounded uniformly discrete metric space which fails the LTP. 
Then it follows from the above and from Theorem~\ref{t:circular} that there exists $\varepsilon>0$ such that every metric $d'$ on $M$ which satisfies
\[
\frac{1}{1+\varepsilon}d(x,y)\leq d'(x,y) \leq (1+\varepsilon)d(x,y)
\]
fails the LTP too. 

We single out a particular example of this fact.
In what follows we will work with the metric graph
$M=\set{0,z}\cup \set{x_i:i\in \Natural}$
where the edges are the couples of the form $\set{0,x_i}$ or $\set{x_i,z}$ and the metric $d$ is the shortest path distance.
\begin{lemma}\label{l:distortion}
Let $d'$ be a metric on $M$ such that $(M,d)$ and $(M,d')$ are Lipschitz equivalent with distortion $D<2$. 
Then $(M,d')$ fails the LTP.
\end{lemma}
Notice that the countable equilateral space is $2$-Lipschitz equivalent to $(M,d)$ so the above lemma is optimal.
\begin{proof}
By the hypothesis there are $D<2$ and $s>0$ such that 
\[
\frac{s}{D}d(x,y)\leq d'(x,y)\leq sd(x,y)
\]
for all $x,y\in M$.
Since the LTP is invariant under scaling of the metric, we may assume that $s=1$.
We are going to show that for $N=\set{0,z}$, $0<\varepsilon<1-\frac{D}2$ and all $u,v \in M$ we have
\[A:=(1-\varepsilon)(d'(0,z)+d'(u,v))>\min\set{d'(0,u)+d'(z,v),d'(0,v)+d'(z,u)}=:B.\]

When $(u,v)=(x_n,x_m)$ we have $A> 2\geq B$.
When $(u,v)=(0,x_n)$ we have $A> \frac{3}{2}>1\geq B$.
The same relation holds when $(u,v)=(z,x_n)$.
\end{proof}

\begin{proposition}\label{p:ellp}
For every $1<p<\infty$, the above space $(M,d)$ embeds into $\ell_p$ with distortion $D<2$. Consequently, $\ell_p$ contains a subset $A$ failing the LTP.
\end{proposition}

\begin{proof}
Let $1<p<\infty$ be fixed.
We define $\phi: M \to \ell_p$ as $\phi(0)=-e_1$, $\phi(z)=e_1$ and $\phi(x_i)=2^{\frac{p-1}{p}}e_i$, where $(e_i)$ is the canonical basis of $\ell_p$.
A routine computation shows that
the distortion of $\phi$ is $\sqrt[p]{1+2^{p-1}}$ which is strictly less than $2$ for $p>1$.
It follows from Lemma~\ref{l:distortion} that $(\phi(M),\norm{\cdot}_p)$ fails the LTP.
\end{proof}

\begin{remark}
Since the $\Real$-trees are exactly those metric spaces which are $CAT(\kappa)$ for every $\kappa \in \Real$ (see~\cite{BH} for this notion)  and since all the infinite subsets of an $\Real$-tree enjoy the LTP, one might be tempted to conjecture that the infinite subsets of $CAT(\kappa)$ spaces have the LTP (at least when $\kappa<0$). 

This turns out not to be the case.
We argue as follows. 
Notice that $\ell_2$ is a $CAT(0)$ space.
By Proposition~\ref{p:ellp}, it contains the set $\phi(M)$ failing the LTP.
By the discussion preceding Lemma~\ref{l:distortion} and using the scaling invariance of the LTP, there is $\varepsilon>0$ such that, if a metric $d'$ on $\phi(M)$ satisfies 
\[
\frac{s}{1+\varepsilon}\norm{x-y}_2\leq d'(x,y)\leq (1+\varepsilon)s\norm{x-y}_2
\]
for some $s>0$ and all $x,y \in \phi(M)$,
then $(\phi(M),d')$ fails the LTP.
We consider $\mathbb{H}^\infty=\set{x \in \ell_2:x_1^2-\sum_{i=2}^\infty x_i^2=1}$ with its hyperbolic distance $\rho$, which is a $CAT(-1)$ space. 
Moreover, it is a Hilbert manifold and so the metric restricted to small enough neighbourhoods of points is as close to the $\ell_2$ metric as one might wish.
In particular there exists an $s>0$ such that $s\closedball{\ell_2}$ embeds into $\mathbb{H}^\infty$ with distortion $1+\varepsilon$, say via a mapping $f$.
We define $d'(x,y):=\rho(f(\frac{s}2x),f(\frac{s}2y))$ for all $x,y \in \phi(M)$.
By the above discussion $(\phi(M),d')$ fails the LTP and is isometric to a subset of $\mathbb{H}^\infty$.
Now, since $\left(\mathbb{H}^\infty,\frac{\rho}{\sqrt{\abs{\kappa}}}\right)$ is a $CAT(\kappa)$ space for $\kappa<0$ we see, by the scaling invariance of the LPT, that the $CAT(\kappa)$ condition does not exclude the presence of the subsets without the LTP.
\end{remark}

The distortion of the embedding in Proposition~\ref{p:ellp} tends to $2$ when $p \to \infty$ or $p \to 1$.
In the case of $p \to \infty$, this is not of a fundamental importance. Indeed, one can easily embed isometrically $(M,d)$ into $c$, the space of convergent sequences.
Similarly, one can easily embed isometrically the space from the Example~\ref{ejenega} into $c_0$.
Thus both $c$ and $c_0$ contain subsets failing the LTP.

On the other hand the behaviour of the distortion when $p \to 1$ is a manifestation of a fundamental fact that we will present next.

We need to introduce the following concepts.
Given a Banach space $(X,\norm{\cdot})$ it is said that $X$ is \textit{asymptotically uniformly convex (AUC)} if, for every $t>0$, the following inequality holds
$$\delta_X(t):=\inf\limits_{x\in S_X}\sup\limits_{codim(Y)<\infty}\inf\limits_{y\in S_Y} \Vert x+ty\Vert-1>0.$$
The function $\delta_X$ is called the \textit{modulus of asymptotic uniform convexity} of $X$ and it has been introduced in~\cite{Milman} (see also \cite{JLPS} for some further properties of this modulus).
It is clear that $\delta_X(t)\leq t$ holds for every $t> 0$. Moreover, $X=\ell_1$ satisfies that $\delta_X(t)=t$ for all $t\in \mathbb R^+$.

\begin{proposition}\label{propauc}
Let $X$ be a AUC Banach space such that $\delta_X(t)=t$ holds for all $t\geq 0$. Then every infinite subset of $X$ has the LTP. 

In particular, for every infinite subset $M$ of $X$ it follows that $\mathcal F(M)$ has an octahedral norm.
\end{proposition}

This proposition generalises Example~\ref{ejepropiP}~(\ref{eje3}) as every $\Real$-tree isometrically embeds into $\ell_1$ of the corresponding density.
Even though the last claim seems to be quite natural, the only proof we know of is  in~\cite{JLPStrees} where it is proved in Proposition 4.1 that the notion of a separable $\Real$-tree coincides with the notion ``SMT'' introduced in that paper. 
It is proved in~\cite[Corollary~2.1]{JLPStrees} that every SMT embeds isometrically into $\ell_1$.
The non-separable case follows the same lines, using transfinite induction.

In the proof of Proposition~\ref{propauc} we shall need the following lemma.

\begin{lemma}\label{cuentasauc}
Let $X$ be a Banach space such that $\delta_X(1)=1$. Then, for every $x\in X$ and every $\varepsilon>0$ there exists a finite-codimensional subspace $Y\subseteq X$ such that, for every $y\in Y$, it follows
$$\Vert x+y\Vert\geq (1-\varepsilon)(\Vert x\Vert +\Vert y\Vert)\ \mbox{for all } y\in Y.$$
In particular, $\delta_X(t)=t$ holds for every $t>0$.
\end{lemma}
\begin{proof}
Pick $x\in X\setminus\{0\}$ and $\varepsilon>0$. Since $\delta_X(1)=1$, then there exists a finite-codimensional subspace $Y$ of $X$ such that, for every $y\in S_Y$, it follows
$$\left\Vert \frac{x}{\Vert x\Vert}+y\right\Vert\geq 2-\varepsilon.$$
Call $z=\frac{x}{\Vert x\Vert}$. Consider $t_1, t_2\in [0,1]$ such that $t_1+t_2=1$ and assume, with no loss of generality, that $t_1\geq t_2$. Then
\[\begin{split}
\Vert t_1 y+t_2 z\Vert& =\Vert t_1(z+y)+(t_2-t_1)z\Vert   \geq t_1\Vert z+y\Vert-(t_1-t_2)\\
& >t_1(2-\varepsilon)+t_2-t_1 \geq t_1+t_2-\varepsilon=1-\varepsilon.
\end{split} \]
Finally, given $y\in Y$, from the previous estimates we get
$$\frac{\Vert  x+y\Vert}{ \Vert x\Vert+\Vert y\Vert}=\left\Vert \frac{ \Vert x\Vert}{ \Vert x\Vert+\Vert y\Vert}\frac{ x}{\Vert x\Vert}+\frac{\Vert y\Vert}{\Vert x\Vert+\Vert y\Vert }\frac{y}{\Vert y\Vert} \right\Vert\geq 1-\varepsilon,$$
and the lemma follows.
\end{proof}

\begin{proof}[Proof of Proposition~\ref{propauc}]
In order to get a contradiction assume that there exists an infinite subset $A\subseteq X$ failing the LTP. By Proposition \ref{propiramsey} we can assume, with no loss of generality, that there are $\varepsilon_0>0$ and $x\neq y\in A$ such that, for every $u\neq v\in A$, we get
$$(1-\varepsilon_0)(\Vert x-y\Vert+\Vert u-v\Vert)>\min\{ \Vert x-u\Vert+\Vert y-v\Vert, \Vert x-v\Vert+\Vert y-u\Vert\}.$$
Since $\delta_X(1)=1$ we conclude the existence of a finite-codimensional subspace $Y\subseteq X$ such that, for all $z\in Y$, it follows
$$\Vert x-y+z\Vert\geq (1-\varepsilon)(\Vert x-y\Vert+\Vert z\Vert),$$
where $0<\varepsilon<\varepsilon_0$. Since $Y$ is finite-codimensional in $X$ we can find a finite-dimensional subspace $F\subseteq X$ such that $X=Y\oplus F$. Consider $P$ and $Q$ to be the corresponding linear and continuous projections onto $Y$ and $F$ respectively. Note that, since $F$ is finite-dimensional, $Q$ is bounded and $A$ is bounded then we can find $B\subseteq A$ such that, for every $u\neq v\in B$, we have that $\Vert Q(u-v)\Vert<\frac{\varepsilon_0-\varepsilon}{4}\Vert x-y\Vert$. Now, for fixed $u\neq v$ in $B$, we have
$$(1-\varepsilon_0)(\Vert x-y\Vert+\Vert u-v\Vert)>\min\{ \Vert x-u\Vert+\Vert y-v\Vert, \Vert x-v\Vert+\Vert y-u\Vert\}.$$
We can assume, with no loss of generality, that the following inequality holds:
\begin{equation}\label{primeineAUC}
(1-\varepsilon_0)(\Vert x-y\Vert+\Vert u-v\Vert)>\Vert x-u\Vert+\Vert y-v\Vert.
\end{equation}
Now
$$\Vert x-u\Vert+\Vert y-v\Vert\geq \Vert x-y-(u-v)\Vert=\Vert x-y-P(u-v)-Q(u-v)\Vert$$
$$\geq  \Vert x-y-P(u-v)\Vert-\Vert Q(u-v)\Vert.$$
Since $P(u-v)\in Y$ we conclude that $\Vert x-y-P(u-v)\Vert>(1-\varepsilon)(\Vert x-y\Vert+\Vert P(u-v)\Vert)$. Consequently, using this joint to (\ref{primeineAUC}), we get
$$(1-\varepsilon_0)(\Vert x-y\Vert+\Vert u-v\Vert)>(1-\varepsilon)(\Vert x-y\Vert+\Vert P(u-v)\Vert)-\Vert Q(u-v)\Vert.$$
Now, the triangle inequality implies that $\Vert u-v\Vert\leq \Vert P(u-v)\Vert+\Vert Q(u-v)\Vert$. Consequently, the previous inequalities imply
$$0\geq (\varepsilon_0-\varepsilon)(\Vert x-y\Vert+\Vert P(u-v)\Vert)-2\Vert Q(u-v)\Vert>\frac{\varepsilon_0-\varepsilon}{2}\Vert x-y\Vert,$$
which is a contradiction. Consequently, we conclude that there exists no subset $A$ of $X$ failing the LTP, so we are done.\end{proof}

\begin{remark}
Given a Banach space $X$ such that $\delta_X(1)=1$ then the metric graph $(M,d)$ can not be embedded in $X$ with distortion $D<2$ as a consequence of Lemma \ref{l:distortion}. We do not know whether this property actually characterises all the AUC Banach spaces with maximal modulus. We do not know either whether the converse of Proposition \ref{propauc} holds.
\end{remark}

\bigskip

\textbf{Acknowledgements}: 
We thank Luis Carlos Garc\'\i a Lirola for frequent conversations about this project.

The first author is grateful to Departamento de An\'alisis Metem\'atico de Universidad de Granada and the Instituto de Matem\'aticas de la Universidad de Granada (IEMath-GR) for the excellent working conditions during his visit in November 2016.

Part of this work was done when the second author visited the Laboratoire de Math\'ematiques de Besan\c{c}on, for which he was supported by a grant from Vicerrectorado de Internacionalizaci\'on y Vicerrectorado de Investigaci\'on y Transferencia de la Universidad de Granada, Escuela Internacional de Posgrado de la Universidad de Granada y el Campus de Excelencia Internacional (CEI) BioTic.

The second author is deeply grateful to the Laboratoire de Math\'ematiques de Besan\c{c}on for their hospitality during the stay.

\end{document}